\documentclass[a4paper,12pt]{article}

\usepackage{mathtools}
\usepackage{amsmath,amssymb,amsthm, bm}
\usepackage{enumerate}
\usepackage{marginfix}
\usepackage{cite}
\usepackage{xcolor}

\let\oldmarginpar\marginpar
\renewcommand\marginpar[1]{\-\oldmarginpar[\raggedleft\tiny #1]%
{\raggedright\tiny #1}}

\usepackage[spanish, english]{babel}
\usepackage[ansinew]{inputenc}
\usepackage{here}

\newcommand{\A}{\mathcal{A}}
\newcommand{\BB}{\mathcal{B}}
\newcommand{\CC}{\mathcal{C}}
\newcommand{\FF}{\mathcal{F}}

\newcommand{\RR}{\mathcal{R}}
\newcommand{\SSS}{\mathcal{S}}
\newcommand{\XX}{\mathcal{X}}

\newcommand{\R}{\mathbb{R}}

\newcommand{\lleq}{\preccurlyeq}
\newcommand{\lle}{\trianglelefteq}
  
\newcommand{\RRR}{R^{>0}\times R^{>0}}
\newcommand{\dom}{ (0,\infty) }

\newcommand{\fS}[2]{(\dagger)_{(#1,#2)}}

\theoremstyle{definition}
\newtheorem{definition}{Definition}[]
\newtheorem{remark}[definition]{Remark}
\newtheorem{example}[definition]{Example}

\theoremstyle{plain}
\newtheorem{theorem}[definition]{Theorem}
\newtheorem{lemma}[definition]{Lemma}
\newtheorem{proposition}[definition]{Proposition}
\newtheorem{corollary}[definition]{Corollary}

\newtheorem{fact}[definition]{Fact}

\title{Directed sets and topological spaces definable in o-minimal structures.}

\author{
Pablo And\'ujar Guerrero\footnote{Department of Mathematics, Purdue University, 150 N. University Street, West Lafayette, IN 47907-2067, U.S.A. \emph{E-mail addresses: pa377@cantab.net (And\' ujar Guerrero), margaret.thomas@wolfson.oxon.org (Thomas)}}
\and 
Margaret E. M. Thomas \footnotemark[1] 
\and 
Erik Walsberg\footnote{Department of Mathematics, Statistics, and Computer Science,
Department of Mathematics, University of California, Irvine, 340 Rowland Hall (Bldg.\# 400),
Irvine, CA 92697-3875, U.S.A. \emph{E-mail address: ewalsber@uci.edu}}
}

\date{}

\begin{document}
\maketitle 

\noindent
{\small \emph{2010 Mathematics Subject Classification.} 03C64 (Primary), 54A20, 54A05, 54D30 (Secondary). }

\begin{abstract}
We study directed sets definable in o-minimal structures, showing that in expansions of ordered fields these admit cofinal definable curves, as well as a suitable analogue in expansions of ordered groups, and furthermore that no analogue holds in full generality. We use the theory of tame pairs to extend the results in the field case to definable families of sets with the finite intersection property. We then apply our results to the study of definable topologies. We prove that all definable topological spaces display properties akin to first countability, and give several characterizations of a notion of definable compactness due to Peterzil and Steinhorn~\cite{pet_stein_99} generalized to this setting.

\end{abstract}

\section{Introduction}

The study of objects definable in o-minimal structures is motivated by the notion that o-minimality provides a rich but \textquotedblleft tame" setting for the theories of said objects.
In this paper we study directed sets definable in o-minimal structures, focusing on expansions of groups and fields. By \textquotedblleft directed set" we mean a preordered set in which every finite subset has a lower (if downward directed) or upper (if upward directed) bound. Our main result (Theorem \ref{them_1st_countability}) establishes the existence of certain definable cofinal maps into these sets. For papers treating similar objects, namely orders and partial orders, the reader may consult~\cite{ram13} and~\cite{ram_stein_14}, in which the authors prove, respectively, that definable orders in o-minimal expansions of groups are lexicographic orders (up to definable order-isomorphism), and that definable partial orders in o-minimal structures are extendable to definable total orders. We do not however use these results in this paper.


The motivation for this paper is not the study of definable directed sets per se, but rather their applications to the theory of definable topological spaces (see Definition~\ref{def_definable_top_spaces}). Our study of such spaces in the o-minimal setting arose from earlier work of the second and third authors on definable function spaces~\cite{thomas12} and definable metric spaces~\cite{walsberg15} respectively, and we give a detailed study of one-dimensional definable topological spaces in~\cite{atw2}. Definable topological spaces have also been studied independently by Johnson~\cite{johnson14}, Fornasiero~\cite{fornasiero}, and Peterzil and Rosel~\cite{pet_rosel_18}. The main result here regarding definable directed sets and its corollaries allow us to derive properties of definable topological spaces that can be interpreted as first countability and compactness in a first-order model theoretic context. 


In Section~\ref{section_definitions} we introduce necessary definitions and conventions. In Section~\ref{section_main_result} we prove the main result on definable directed sets (Theorem \ref{them_1st_countability}) and show that it does not hold in all o-minimal structures. 
In Section~\ref{section_after_main_result} we strengthen the main result in the case where the underlying structure expands an ordered field. In Section~\ref{section_tame_pairs} we apply the theory of tame pairs initiated in ~\cite{mark_stein_94} to make some additional remarks and frame our work in the context of types. We also strengthen our earlier results in the case where the underlying structure expands an archimedean field. In Section~\ref{section_def_top_spaces} we use the results in previous sections to describe definable bases of neighborhoods of points in a definable topological space (definable first countability) and derive some consequences of this, in particular showing that, whenever the underlying o-minimal structure expands an ordered field, definable topological spaces admit definable curve selection (without the condition that the curves be continuous). Finally, in Section~\ref{section_compactness} we consider the standard notion of definable compactness in the o-minimal setting due to Peterzil and Steinhorn~\cite{pet_stein_99} and show how it relates to other properties that could also be understood to capture compactness in the definable context. 
We conclude by showing that, whenever the underlying o-minimal structure expands the field of reals, a definable topological space is definably compact if and only if it is compact. 

\paragraph{Acknowledgements}
The first and second authors were supported by German Research Council (DFG) Grant TH 1781/2-1; the Zukunftskolleg, Universit\"at Konstanz; and the Fields Institute for Research in Mathematical Sciences, Toronto, Canada (during the Thematic Program on \textquotedblleft Unlikely Intersections, Heights and Efficient Congruencing"). Additionally the first author was supported by the Canada Natural Sciences and Engineering Research Council (NSERC) Discovery Grant RGPIN-06555-2018; and the second author by the Ontario Baden--W\"urttemberg Foundation, and the Canada Natural Sciences and Engineering Research Council (NSERC) Discovery Grant RGPIN 261961.

The third author was partially supported by the European Research Council under the European Unions Seventh Framework Programme (FP7/2007-2013) / ERC Grant agreement no. 291111/ MODAG.

We would like to thank the referee for careful reading of the paper and detailed feedback which enhanced the presentation.

\section{Definitions}\label{section_definitions}

Throughout, unless otherwise indicated, $\RR=(R,0,+,<,\ldots)$ denotes an o-minimal expansion of an ordered group, and by \textquotedblleft definable" we mean \textquotedblleft $\RR$-definable, possibly with parameters from $R$".
In addition, $m,n,N$ denote natural numbers and $k,l$ denote integers.
By \textquotedblleft interval" we mean an interval with distinct endpoints in $R\cup \{\pm \infty\}$ (i.e. definable and not a singleton). The euclidean topology on $R^n$ is the product topology given the order topology on $R$. 
We let $\|\cdot\|:R^n\rightarrow R$ be the usual $l_\infty$ norm
$$ \| (x_1,\ldots,x_n)\| = \max \{ |x_1|, \ldots, |x_n|\}.$$
For any $x\in R^n$ and $\varepsilon>0$, set 
$$B(x,\varepsilon):=\{y\in R^n : \|x-y\|<\varepsilon\}$$
 and 
 $$\overline{B(x,\varepsilon)}:=\{y\in R^n : \|x-y\|\leq\varepsilon\}$$ 
 to be respectively the open and closed ball of center $x$ and radius $\varepsilon$ . 
 A set $S\subseteq R^n$ is bounded if it is bounded in the euclidean topology, i.e. if it lies inside some ball. When we write that $\SSS=\{S_u : u\in\Omega\}$ is a definable family of sets we are implicitly associating $\SSS$ with the definable set $\Omega$ and some formula $\phi(u,v)$ such that, for every $u\in \Omega$, $S_u = \{ v : \RR\models \phi(u,v)\}$.   

Recall that a preorder is a transitive and reflexive binary relation, and that a downward directed set $(\Omega,\lleq)$ is a nonempty set $\Omega$ with a preorder $\lleq$ satisfying that, for every finite subset $\Omega'\subseteq \Omega$, there exists $v\in \Omega$ such that $v \lleq u$ for all $u\in \Omega'$. The dual notion of upward directed set is defined analogously. 

\begin{definition}
A \emph{definable preordered set} is a definable set $\Omega\subseteq R^n$ together with a definable preorder $\lleq$ on $\Omega$. A \emph{definable downward directed set} is a definable preordered set $(\Omega, \lleq)$ that is downward directed.   
\end{definition}

Recall that a subset $S$ of a preordered set $(\Omega, \lleq)$ is cofinal if, for every $u\in \Omega$, there exists $v\in S$ such that $u\lleq v$. We refer to this property as upward cofinality, and work mostly with the dual notion of coinitiality, which we in turn refer to as downward cofinality. The reason for this approach is that it seems more natural for the later application of our results to definable topologies.

\begin{definition}
Let $(\Omega', \lleq')$ and $(\Omega, \lleq)$ be preordered sets. Given $S\subseteq \Omega$, a map $\gamma:\Omega'\rightarrow \Omega$ is \emph{downward cofinal for $S$ (with respect to $\lleq'$ and $\lleq$)} if, for every $u\in S$, there exists $v=v(u)\in \Omega'$ such that $w\lleq' v$ implies $\gamma(w)\lleq u$. Equivalently, we say that $\gamma:(\Omega',\lleq')\rightarrow (\Omega,\lleq)$ is \emph{downward cofinal for $S$}. We write this as $\gamma \lleq' \lleq S$ when there is no room for confusion. We say that $\gamma$ is \emph{downward cofinal} if it is downward cofinal for $\Omega$.
   
We say that a map $\gamma:\dom \rightarrow \Omega$ is a \emph{curve in $\Omega$}, and that it is \emph{downward cofinal for $S\subseteq \Omega$} if $\gamma:(\dom,\leq)\rightarrow (\Omega, \lleq)$ is. If there exists such a map when $S=\Omega$, then we may say that $(\Omega, \lleq)$ \emph{admits a downward cofinal curve}. 
\end{definition}

The dual notion of definable downward directed set is that of \emph{definable upward directed set}. 
In other words, if $(\Omega, \lleq)$ is a preordered set and $\lleq^*$ is the dual preorder of $\lleq$, then $(\Omega, \lleq)$ is a definable upward directed set whenever $(\Omega, \lleq^*)$ is a definable downward directed set. 
Moreover, given a preordered set $(\Omega',\lleq')$, a map $\gamma:(\Omega',\lleq')\rightarrow (\Omega, \lleq)$ is \emph{upward cofinal for $S\subseteq \Omega$} if it is downward cofinal for $S$ with respect to $\lleq'$ and $\lleq^*$. 

Note that the image of a downward (respectively upward) cofinal map is always downward (respectively upward) cofinal.

In o-minimality, curves are classically defined to be any map into $R^n$ with interval domain~\cite{dries98}. Our definition of curve is justified because, in our setting (that $\RR$ expands an ordered group), this notion is equivalent for all practical purposes to the weaker one in~\cite{dries98}. Additionally our notion of downward/upward cofinal curve focuses on the behaviour of curves as $t\rightarrow 0^+$. Cofinal curves however will only be relevant in the setting where $\RR$ expands an ordered field, where once again this notion is strong enough for all purposes. 
  

\begin{remark}\label{remark_sets_order}

\begin{enumerate}[(i)]
\item \label{itm:remark_i}
Let $\SSS=\{S_u : u\in \Omega\}$ be a definable family of sets. Set inclusion induces a definable preorder $\lleq_{\SSS}$ on $\Omega$ given by \mbox{$u\lleq_{\SSS} v \Leftrightarrow S_u \subseteq S_v$}. 

\item \label{itm:remark_ii} Conversely, let $(\Omega, \lleq)$ be a definable preordered set. Consider the definable family of nonempty sets $\{S_u : u \in \Omega\}$, where $S_u=\{v\in \Omega : v\lleq u\}$ for every $u\in \Omega$. Then, for every $u,w\in\Omega$, $u\lleq w$ if and only if $S_u\subseteq S_w$. 
\end{enumerate}
\end{remark}
 
Note that Remark~\ref{remark_sets_order} remains true if we drop the word \textquotedblleft definable" from the statements.  
Motivated by Remark~\ref{remark_sets_order}(\ref{itm:remark_i}) we introduce the following definition.

\begin{definition}
A family of sets $\SSS=\{S_u : u\in\Omega\}$ is \emph{downward} (respectively \emph{upward}) \emph{directed} if the preorder $\lleq_\SSS$ on $\Omega$ induced by set inclusion in $\SSS$ forms a downward (respectively upward) directed set. For convenience we also ask that $\SSS$ does not contain the empty set.
\end{definition}

In other words, $\SSS$ is downward (respectively upward) directed if and only if $\emptyset \notin \SSS$ and, for every finite $\FF\subseteq \SSS$, there exists $u=u(\FF)\in \Omega$ satisfying $S_u \subseteq S$ (respectively $S \subseteq S_u$) for every $S\in \FF$. 

\begin{example}\label{ex_FIP}
Let $(\Omega, \lleq)$ be a downward (respectively upward) directed set. The induced family $\SSS=\{S_u : u\in\Omega\}$ described in Remark~\ref{remark_sets_order}(\ref{itm:remark_ii}) is a downward (respectively upward) directed family.
\end{example}

\begin{definition}
Let $(\Omega,\lleq)$ be a definable preordered set and let $f:\Omega\rightarrow R^m$ be a definable injective map. We define the \emph{push-forward of $(\Omega, \lleq)$ by $f$} to be the definable preordered set $(f(\Omega),\lleq_f)$ that satisfies: $f(x)\lleq_f f(y)$ if and only if $x\lleq y$, for all $x,y\in \Omega$. 
Thus $(f(\Omega),\lleq_f)$ is the unique definable preordered set such that  \mbox{$f:(\Omega, \lleq)\rightarrow (f(\Omega), \lleq_f)$} is a preorder-isomorphism. 

Let $\{S_u : u\in \Omega\}$ be a definable family of sets. Based on the correspondence between definable families of sets and definable preordered sets given by Remark~\ref{remark_sets_order} we also define the \emph{push-forward of $\{S_u : u\in\Omega\}$ by $f$} to be the reindexing of said family given by $\{S_{f^{-1}(u)} : u\in f(\Omega)\}$. 
We abuse notation and write $S_u$ instead of $S_{f^{-1}(u)}$ when it is clear that $u\in f(\Omega)$.
\end{definition}


\section{Main result}\label{section_main_result}

\begin{definition}
We equip $\RRR$ with the definable preorder given by
\[
(s,t)\lle (s',t') \Leftrightarrow s \leq s' \text{ and } t \geq t'.
\]
Note that $(\RRR, \lle)$ is a definable downward directed set.
\end{definition}


We are now ready to state the Main Theorem of this paper.

\begin{theorem}\label{them_1st_countability}
Let $(\Omega, \lleq)$ be a definable downward (respectively upward) directed set. There exists a definable downward (respectively upward) cofinal map \mbox{$\gamma:(\RRR,\lle)\rightarrow (\Omega,\lleq)$}. 
\end{theorem}

Note that, if $(\Omega, \lleq)$ is definable over a set of parameters $A$, and $b\in R$ is any non-zero parameter, then the downward (respectively upward) cofinal map given by Theorem~\ref{them_1st_countability} is $Ab$-definable. To see this, it suffices to recall that the structure $(\RR,b)$ has definable choice, where every choice function is $0$-definable. Equivalently any choice function in $\RR$ is $b$-definable. Then apply the theorem to the prime model over $Ab$, which by definable choice is $dcl(Ab)$.


We prove the downward case of Theorem~\ref{them_1st_countability}. From the duality of the definitions it is clear that the upward case follows from this. Therefore, since there will be no room for confusion, from now on and until the end of the section we write \textquotedblleft directed" and \textquotedblleft cofinal" instead of respectively \textquotedblleft downward directed" and \textquotedblleft downward cofinal". 

In order to prove Theorem~\ref{them_1st_countability} we require three lemmas, the first of which concerns definable families with the finite intersection property. 
Recall that a family of sets $\SSS$ has the \emph{Finite Intersection Property} (FIP) if the intersection $\bigcap_{1\leq i\leq n} S_i$ is nonempty for all $S_1,\ldots, S_n\in \SSS$.

\begin{example}
If $\SSS$ is a directed family of sets (e.g. Example \ref{ex_FIP}), then $\SSS$ has the finite intersection property. The converse is not necessarily true; the definable family $\{R\setminus\{x\} : x\in R\}$ has the finite intersection property but is not directed.
\end{example}

\begin{lemma}\label{lemma_f}
Let $\Omega\subseteq R^N$ and let $\SSS=\{S_u\subseteq R^n : u\in \Omega\}$ be a definable family of sets with the finite intersection property. 
There exists a definable set $\Omega_h$ and a definable bijection \mbox{$h:\Omega \rightarrow \Omega_h$} such that the push-forward of $\SSS$ by $h$ satisfies the following properties. 
\begin{enumerate}[(1)]
\item For every $u\in \Omega_h$, there exists $\varepsilon=\varepsilon(u)>0$ such that \mbox{$\bigcap_{v\in\Omega_h, \; \|v-u\|<\varepsilon} S_v \neq \emptyset$.} 
\item For every closed and bounded definable set $B\subseteq \Omega_h$, there exists $\varepsilon=\varepsilon(B)>0$ such that $\bigcap_{v\in B, \; \|v-u\|<\varepsilon} S_v \neq \emptyset$ for every $u\in B$.
\end{enumerate}
\end{lemma}

\begin{proof}
We prove the result by showing that, for $\Omega$ and $\SSS$ as in the statement of the lemma, there exists a definable map \mbox{$f:\Omega \rightarrow \cup_{u\in\Omega} S_u \times R^{>0}$}, given by $u\mapsto (x_u,\varepsilon_u)$, such that, for all $u, v \in \Omega$,
\begin{equation}\label{eqn_f}\tag{$\fS{f}{\SSS}$}
 \|f(u)-f(v)\|<\frac{\varepsilon_u}{2} \Rightarrow x_u \in S_v.
\end{equation} 
If $\fS{f}{\SSS}$ holds for a continuous function $f$, then the lemma holds with $\Omega_h=\Omega$ and $h=id$. This follows from the observation that, by $\fS{f}{\SSS}$ and the continuity of $f$ at $u\in\Omega$, we have 
\[
x_u \in \bigcap_{\substack{v\in\Omega \\ \|u-v\|<\delta}} S_v \neq \emptyset
\]
whenever $\delta>0$ is sufficiently small.
Moreover, if $f$ is continuous on a definable closed and bounded set $B\subseteq \Omega$ and $\varepsilon_m$ is the minimum of the map $u\mapsto \varepsilon_u$ on $B$, then uniform continuity yields a $\delta>0$ such that $\|f(u)-f(v)\|<\frac{\varepsilon_m}{2}$ for all $u,v\in B$ satisfying $\|u-v\|<\delta$. So in this case we have
\[
x_u \in \bigcap_{\substack{v\in B \\ \|u-v\|<\delta}} S_v \neq \emptyset, \quad \text{ for all $u\in B$}. 
\]

In the case that $\fS{f}{\SSS}$ holds for a function $f$ which is not necessarily continuous, then we may modify the above argument by identifying an appropriate bijection $h$ and push-forward $\SSS_h=\{S_{h^{-1}(u)} : u\in\Omega_h\}$ to complete the proof as follows.
By o-minimality, let $C_1,\ldots, C_l$ be a a cell partition of $\Omega$ such that $f$ is continuous on every $C_i$. Consider the disjoint union $\Omega_h=\cup_{1\leq i\leq l} (C_i \times \{i\})$ and the natural bijection $h:\Omega \rightarrow \Omega_h$. This map is clearly definable. Moreover, for every cell $C_i$, the restriction $h|_{C_i}$ is a homeomorphism and $h(C_i)$ is open in $\Omega_h$. It follows that the map $f\circ h^{-1}:\Omega_h\rightarrow \cup_{u\in \Omega} S_u\times R^{>0}$ is definable and continuous. Additionally, note that $\fS{f\circ h^{-1}}{\SSS_h}$ holds. Consequently, the lemma holds via an analogous argument to the one above. 

Thus it remains to prove the existence of a definable function $f$ satisfying $\fS{f}{\SSS}$. 
First of all we show that we can assume that there exists $0\leq m \leq n$ such that the following property (\ref{eqn_P}) holds:
\begin{equation}\label{eqn_P}\tag{P}
\text{dim}\left(\bigcap_{1\leq i\leq k} S_{u_i}\right)=m, \text{ for every } u_1,\ldots, u_k \in \Omega. 
\end{equation}
Let $m'$ be the minimum natural number such that there exist $u_1,\ldots, u_k \in \Omega$ with $\text{dim}(\cap_{1\leq i\leq k} S_{u_i})=m'$. 
Set $S:=\cap_{1\leq i\leq s} S_{u_i}$. Consider the definable family $\SSS_*:=\{S_u \cap S : u\in \Omega\}$. This family has the FIP. Note that if $\fS{f}{\SSS_*}$ is satisfied for some definable function $f$ then $\fS{f}{\SSS}$ is satisfied too. Hence, by passing to $\SSS_*$ if necessary, we may assume that (P) holds for some fixed $m$.

Suppose that $m = n$, so in particular all sets $S_u$ have nonempty interior (we call this the open case). By definable choice, let $f:\Omega \rightarrow \cup_u S_u \times R^{>0}$, $f(u)=(x_u,\varepsilon_u)$, be a definable map such that, for every $u\in\Omega$, the open ball of center $x_u$ and radius $\varepsilon_u$ is contained in $S_u$. Then, for any $u,v\in \Omega$, $\|f(u)-f(v)\|<\frac{\varepsilon_u}{2}$ implies both $\|x_{u}-x_{v}\|<\frac{\varepsilon_u}{2}$ and $\varepsilon_{v}>\frac{\varepsilon_{u}}{2}$. Hence $\|x_{u}-x_{v}\|<\varepsilon_{v}$, and so $x_u \in S_{v}$.

Now suppose that $m<n$. 
Let $u_0$ be a fixed element in $\Omega$ and let $\XX$ be a finite partition of $S_{u_0}$ into cells. We claim that there must exist some cell $C\in \XX$ such that $\text{dim}(C \cap \bigcap_{1\leq i\leq k} S_{u_i})=m$ for any $u_1,\ldots, u_k \in \Omega$. Suppose that the claim is false. Then, for every $C\in \XX$, there exist $k_C<\omega$ and $u^C_{1},\ldots, u^C_{k_C} \in \Omega$ such that $\text{dim}(C \cap \bigcap_{1\leq i\leq k_C} S_{u^C_i})<m$. In that case however 
\begin{align*}
m &=\text{dim} \left( S_{u_0} \cap \bigcap_{C\in\XX, 1\leq i \leq k_C} S_{u^C_i} \right) 
\leq \text{dim} \left( \bigcup_{C\in \XX} \left( C \cap \bigcap_{ 1\leq i \leq k_C} S_{u^C_i}) \right) \right) \\
&=\max_{C\in\XX} \left \{ \text{dim} \left( C \cap \bigcap_{ 1\leq i \leq k_C} S_{u^C_i} \right) \right \} <m,
\end{align*}
which is a contradiction. So the claim holds. 

Let $C_0 \in \XX$ be a cell with the described property. Clearly $\text{dim}(C_0)=m$.
Consider the definable family $\SSS'=\{S'_u=C_0\cap S_u : u\in \Omega\}$. By the claim, $\SSS'$ satisfies the FIP; in fact, any intersection of finitely many sets in $\SSS'$ has dimension $m$. We prove the lemma for $\SSS'$ and the result for $\SSS$ follows. 

Let $\pi:C_0\rightarrow \pi(C_0)\subseteq R^m$ be a projection which homeomorphically maps $C_0$ onto an open cell $\pi(C_0)$.
Since $\pi$ is a bijection, the set $\pi(S'_u)\subseteq R^m$ has nonempty interior, for every $u\in\Omega$. 
Consider the definable family $\{\pi(S'_u)\subseteq R^m : u\in \Omega\}$. 
This is a definable family of sets with nonempty interior that has the FIP. By the open case, there exists a definable $g=(g_1,g_2):\Omega \rightarrow R^m\times R^{>0}$, $g(u)=(x_u,\varepsilon_u)$, such that, for every $u,v\in \Omega$, $\|g(u)-g(v)\|<\frac{\varepsilon_u}{2}$ implies $x_u \in \pi(S'_v)$, i.e. $\pi^{-1}(x_u)\in S'_v$. Let $f=(f_1,f_2):\Omega \rightarrow R^m\times R^{>0}$ be given by $f_1=\pi^{-1}\circ g_1$ and $f_2=g_2$. Since $g$ is a projection of $f$, we have $\|g(u)-g(v)\|\leq \|f(u)-f(v)\|$, for all $u,v\in \Omega$. 
The result follows.
\end{proof}

\begin{definition}
Let $(\Omega, \lleq)$ be a directed set. 
We say that $\Omega' \subseteq \Omega$ is \emph{$\lleq$-bounded in $\Omega$} if there exists $v\in\Omega$ such that $v\lleq u$ for all $u\in \Omega'$. We write $v\lleq \Omega'$. 
\end{definition}


Let $(\Omega,\lleq)$ be a definable directed set and let $\SSS=\{S_u : u\in \Omega\}$ be a directed family of sets as in Example~\ref{ex_FIP}. By construction of $\SSS$, note that, if a subfamily $\{S_u :u\in \Omega'\}$ has nonempty intersection, then $\Omega'$ is $\lleq$-bounded in $\Omega$. Hence Lemma~\ref{lemma_f} yields the following corollary.

\begin{corollary}\label{remark_another_lemma_f}
 Any definable directed set $(\Omega',\lleq')$ is definably preorder-isomorphic to a definable directed set $(\Omega,\lleq)$ such that 
\begin{enumerate}[(1)]
\item  for all $u\in \Omega$ there exists an $\varepsilon>0$ such that $B(u,\varepsilon)\cap \Omega$ is $\lleq$-bounded in $\Omega$;
\item for any definable closed and bounded set $B\subseteq \Omega$, there exists an $\varepsilon>0$ such that $B(u,\varepsilon)\cap B$ is $\lleq$-bounded in $\Omega$ for every $u\in B$.
\end{enumerate}
\end{corollary}

We now prove a lemma (Lemma~\ref{lemma_2_preorders}) which allows us to see $\lleq$-boundedness as a local property, and which we will use in proving Theorem~\ref{them_1st_countability}. 
We first show that the doubling property of the supremum metric in $\R^n$ generalises to $R^n$.

\begin{lemma}\label{lemma_doubling}
For any $x\in R^n$ and $r>0$, there exists a finite set of points $P\subseteq B(x,r)$, where $|P|\leq 3^n$, such that the sets of balls of radius $\frac{r}{2}$ centered on points in $P$ covers the ball of radius $r$ centered on $x$, i.e.
\[
B(x,r)\subseteq \bigcup_{y\in P} B\left( y, \frac{r}{2} \right).
\] 
\end{lemma}
\begin{proof}
Fix $x=(x_1,\ldots, x_n)\in R^n$ and $r>0$. Set 
\[
P := \left\{ y=(y_1,\ldots, y_n): y_i=x_i+\delta_i,\, \delta_i\in\left\{\frac{r}{2}, -\frac{r}{2}, 0 \right\}, \,  1\leq i \leq n \right\}.
\]
Note that $|P|=3^n$ and, for every $z\in B(x,r)$, there exists some $y\in P$ such that $\|y-z\|<\frac{r}{2}$. 
Thus $B(x,r)\subseteq \bigcup_{y\in P} B(y,\frac{r}{2})$.  
\end{proof}



\begin{lemma}\label{lemma_2_preorders}
Let $(\Omega, \lleq)$ be a definable directed set and let $S\subseteq \Omega$ be a bounded definable set. Suppose that there exists $\varepsilon_0>0$ such that, for all $u\in S$, $B(u,\varepsilon_0)\cap S$ is $\lleq$-bounded in $\Omega$. Then $S$ is $\lleq$-bounded in $\Omega$.
\end{lemma}
\begin{proof}

Consider the definable set 
$$H=\{\varepsilon : \forall u\in S, \, B(u,\varepsilon)\cap S \text{ is $\lleq$-bounded in $\Omega$}\}.$$
We have $(0,\varepsilon_0)\subseteq H$, and so $H$ is nonempty, and hence by o-minimality it must be of the form $(0,r)$, for some $r\in [\varepsilon_0,\infty) \cup \{\infty\}$.
 We show $H=(0,\infty)$ and thus, since $S$ is bounded, that $S$ is $\lleq$-bounded in $\Omega$. 

Suppose that $r < \infty$.
 We reach a contradiction by showing $\frac{4}{3}r \in H$. To do so we fix an arbitrary $u_0 \in S$ and show that $B(u_0, \frac{4}{3}r)\cap S$ is $\lleq$-bounded in $\Omega$. 

By repeated application of Lemma~\ref{lemma_doubling}, the ball $B(u_0, \frac{4}{3}r)$ can be covered by finitely many ($\leq 3^{2n}$) balls $B_1, \ldots, B_{k}$ of radius $\frac{r}{3}$.
 For any $1\leq i \leq k$, if $w\in B_i \cap S$, then $B_i\subseteq B(w,\frac{2}{3}r)$, by the triangle inequality. By assumption, $B(w,\frac{2}{3}r)\cap S$ is $\lleq$-bounded in $\Omega$, and so the set $B_i \cap S$ is $\lleq$-bounded in $\Omega$. 
Moreover, if $B_i \cap S$ is empty for some for $1\leq i \leq k$, then $B_i \cap S$ is trivially $\lleq$-bounded in $\Omega$. By the definition of directed set, it follows that $\cup_i B_i \cap S$ is $\lleq$-bounded in $\Omega$, and so, since $B(u_0,\frac{4}{3}r) \subseteq \cup_i B_i$, the set $B(u_0, \frac{4}{3} r)\cap S$ is $\lleq$-bounded in $\Omega$. 
\end{proof}

Corollary~\ref{remark_another_lemma_f} and Lemma~\ref{lemma_2_preorders} together yield the following.
 
\begin{corollary}\label{cor_for_proof_main_them}
Let $(\Omega',\lleq')$ be a definable directed set. Then there exists a definable directed set $(\Omega, \lleq)$ that is definably preorder isomorphic to $(\Omega,\lleq)$ and such that every closed and bounded definable subset of $\Omega$ is $\lleq$-bounded in $\Omega$. 
\end{corollary}
 
We now borrow a definition from~\cite{dolich_miller_steinhorn_10}.
\begin{definition}
A definable set $S$ is \emph{$D_\Sigma$} if there exists a definable family of closed and bounded sets $\{S(s,t) : (s,t) \in \RRR\}$ such that $S=\bigcup_{s,t} S(s,t)$ and $S(s,t)\subseteq S(s',t')$ whenever $(s',t')\lle (s,t)$. 
\end{definition}

To prove Theorem~\ref{them_1st_countability} we use the fact that every definable set is $D_\Sigma$. This can be derived from~\cite{dolich_miller_steinhorn_10}.
We include a proof for the sake of completeness.

\begin{lemma}\label{lemma_D_sigma}
Every definable set is $D_\Sigma$.
\end{lemma}
\begin{proof}
Let $S$ denote a definable set. We proceed by induction on $\text{dim}(S)$.

If $S$ is closed, then define $S(s,t)=\overline{B(0,t)}\cap S$, for every $(s,t)\in\RRR$. Clearly every $S(s,t)$ is closed and bounded and $\cup_{s,t} S(s,t)=S$. Moreover, $S(s,t)\subseteq S(s',t')$ whenever $t\leq t'$. So $S$ is $D_\Sigma$. In particular, if $\dim(S)\leq 0$, then $S$ is $D_\Sigma$. 

Now suppose that $\dim(S)\geq 1$. By the above, we may assume that $S$ is not closed. Set $\overline{\partial S}:=cl(\partial S)$. By o-minimality, $\dim(\overline{\partial S})< \dim(S)$ and so $S \cap \overline{\partial S}$ is $D_\Sigma$ by induction. Let $S_0=S\setminus \overline{\partial S}$. Since the union of finitely many closed and bounded sets is closed and bounded, one may easily deduce that the union of finitely many $D_\Sigma$ sets is $D_\Sigma$. Hence to prove that $S$ is $D_\Sigma$ is suffices to show that $S_0$ is $D_\Sigma$. 

For every $s>0$, set 
$$\overline{\partial S}(s):=\bigcup\{B(x,s) : x\in \overline{\partial S}\}.$$
 Note that, since $\partial S_0 \subseteq cl(S)\setminus S_0=\overline{\partial S}$, for every $s>0$ it holds that $S_0\setminus \overline{\partial S}(s)=cl(S_0)\setminus \overline{\partial S}(s)$, meaning that $S_0\setminus \overline{\partial S}(s)$ is closed. For any $(s,t)\in\RRR$, let $S(s,t)=\overline{B(0,t)}\cap S_0 \setminus \overline{\partial S}(s)$. Every set $S(s,t)$ is closed and bounded and $S(s,t)\subseteq S(s',t')$ whenever $(s',t')\lle (s,t)$. Moreover, since $\overline{\partial S}$ is closed, for every $x\in S_0$ there exists $s>0$ such that $B(x,s)\cap  \overline{\partial S}=\emptyset$, so $x\notin \overline{\partial S}(s)$, and in particular $x\in S(s,t)$ for all sufficiently large $t>0$.
 Hence $\cup_{s,t} S(s,t)=S$. So $S_0$ is $D_\Sigma$.  
\end{proof}

We now complete the proof of Theorem~\ref{them_1st_countability}. 

\begin{proof}[Proof of Theorem~\ref{them_1st_countability}]
Let $(\Omega,\lleq)$ be a definable directed set. 
We construct a definable cofinal map $\gamma:(\RRR,\lle)\rightarrow (\Omega,\lleq)$. 

Clearly it is enough to prove the statement for any definable preorder definably preorder-isomorphic to $(\Omega,\lleq)$. Hence, by Corollary~\ref{cor_for_proof_main_them}, we may assume that any definable closed and bounded subset of $\Omega$ is $\lleq$-bounded. 

Lemma~\ref{lemma_D_sigma} yields a definable family $\{ \Omega(s,t) : (s,t)\in \RRR\}$ of closed and bounded sets such that $\Omega=\cup_{s,t} \Omega(s,t)$ and $\Omega(s,t)\subseteq \Omega(s',t')$ whenever $(s',t')\lle (s,t)$. By assumption on $\Omega$, every $\Omega(s,t)$ is $\lleq$-bounded. 
Applying definable choice let $\gamma: \RRR\rightarrow \Omega$ be a definable map satisfying $\gamma(s,t)\lleq \Omega(s,t)$, for every $(s,t)\in\RRR$ (if $\Omega(s,t)$ is empty then trivially any value $\gamma(s,t)\in \Omega$ will do).
 For every $x\in \Omega$, there exists $(s_x,t_x)\in \RRR$ such that $x\in \Omega(s_x,t_x)\subseteq \Omega(s,t)$, for all $(s,t) \lle (s_x,t_x)$. We conclude that $\gamma$ is cofinal.
\end{proof} 

Note that Theorem~\ref{them_1st_countability} implies that, if $(R,<)$ is separable, then every definable directed set has countable cofinality. In Proposition~\ref{prop_1st_countability} we show that this holds in the greater generality of any o-minimal structure.

\begin{remark}
Given a definable directed set $(\Omega,\lleq)$, one may ask whether or not the map $\gamma: (\RRR,\lle)\rightarrow (\Omega, \lle)$ given by Theorem~\ref{them_1st_countability} may always be chosen such that $\gamma(\RRR)$ is totally ordered by $\lleq$.
The answer is no. Consider the definable family $\{ (0,t)\cup (2t,3t) : t>0 \}$. Following Remark~\ref{remark_sets_order}(\ref{itm:remark_i}), set inclusion in this family defines a directed set $(R^{>0},\lleq)$, where $t \lleq t'$ iff $t=t'$ or $3t\leq t'$. 
It is easy to see that no infinite definable subset of $R^{>0}$ is totally ordered by $\lleq$. 
\end{remark}

Applying first-order compactness and definable choice in the usual fashion we may derive a uniform version of Theorem~\ref{them_1st_countability}.
We leave the details to the reader.

\begin{corollary}\label{cor_improvem_them}
Let $\{(\Omega_x,\lleq_x) : \Omega_x\subseteq R^m, x\in\Sigma\subseteq R^n\}$ be a definable family of downward directed sets. There exists a definable family of functions $\{\gamma_x:(\RRR,\lle)\rightarrow (\Omega_x,\lleq_x) :x\in\Sigma\}$ such that, for every $x\in\Sigma$, $\gamma_x$ is downward cofinal. 
\end{corollary}

We end this section with an example of an o-minimal structure in which Theorem~\ref{them_1st_countability} does not hold. For the rest of the section we drop the assumption that $\RR$ expands a group but keep the assumption that it is an o-minimal expansion of a dense linear order without endpoints. We will consider the property that any definable function from $R$ to itself is piecewise either constant or the identity (think of $\RR$ as having trivial definable closure, say; for example $\RR=(R,<)$ is a dense linear order without endpoints).

Theorem~\ref{them_1st_countability} tells us that, under the assumption that $\RR$ expands an ordered group, any definable upward directed set has a definable cofinal subset of dimension at most $2$ (and analogously for definable downward directed sets). 




For every $n > 0$, we prove, under the assumption of the above property on $\RR$, the existence of a definable upward directed set that admits no definable cofinal subset of dimension less than $n$. Thus Theorem~\ref{them_1st_countability} does not hold in general for o-minimal structures, even if we substitute $(\RRR,\lle)$ with any other definable preordered set.  

We begin with some notation. For any $a\in R$ and $n > 0$, let 
$$X(a,n) = \{(x_1,\ldots, x_n)\in R^n : a<x_1<\cdots<x_n \}.$$
Then $X(a,n)$ is definable and $n$-dimensional. We endow $X(a,n)$ with the definable lexicographic order and show that any definable cofinal subset of $X(a,n)$ has dimension $n$. Hence for the remainder of the section `cofinal' will refer to the lexicographic order. 

\begin{proposition}\label{prop:example}
Suppose that any definable map from $R$ to itself is piecewise constant or the identity. 
Then any cofinal definable subset of $X(a,n)$ is $n$-dimensional.
\end{proposition}

Let $\pi$ denote the projection onto the first coordinate. We will make use of the following easily derivable fact:

\begin{fact}\label{fact_lex}
For any $m > 0$ and $b\in R$, a definable subset $Y\subseteq X(b,m)$ is cofinal if and only if $(r,\infty)\subseteq\pi(Y)$ for some $r>0$.   
\end{fact}

We now prove Proposition~\ref{prop:example}.

\begin{proof}
We proceed by induction on $n$.

Fix $a\in R$ and $n > 0$ and let $X=X(a,n)$. If $n=1$ then the result is trivial. For the inductive step suppose that $n>1$. We assume that there exists a definable cofinal subset $Y\subset X$ with $\dim Y <n$ and arrive at a contradiction. 

Fact~\ref{fact_lex} implies that $Y_t=\{x\in R^{n-1} : (t,x)\in Y\}$ must be nonempty for sufficiently large $t>0$. 
As $\dim Y <n$, the fiber lemma for o-minimal dimension (Proposition $1.5$ Chapter $4$ in~\cite{dries98}) implies the existence of some $t_0>0$ such that $0\leq \dim Y_t < n-1$, for all $t>t_0$. 
Now note that, for every $t>a$, the fiber $X_t$ is the space $X(t,n-1)$. 
Applying induction we conclude that $Y_t$ is not cofinal in $X_t$ for any $t>t_0$.
 By Fact~\ref{fact_lex} and the fact that $Y_t$ is nonempty, it follows that $\pi(Y_t)$ has a supremum, which clearly must be greater than $t$. Consider the definable map $f:(t_0,\infty)\rightarrow R$ given by $t\mapsto \sup \pi(Y_t)$. This map is well defined and, for every $t$, $f(t)>t$. This contradicts the fact that $f$ is piecewise constant or the identity. 
\end{proof}

\section{A strengthening of the main result for expansions of ordered fields}\label{section_after_main_result}

We recall that a pole is a definable bijection $f : I \to J$ between a bounded interval $I$ and an unbounded interval $J$.
Edmundo proved the following in~\cite{edmundo00} (see Fact 1.6).
 
\begin{fact}\label{fact:edmundo}
The structure $\RR$ has a pole if and only if there exist definable functions $\oplus, \otimes : R^2 \to R$ such that $(R,<,\oplus, \otimes)$ is a real closed ordered field.
\end{fact} 

The following corollary provides a condition on the structure $\RR$ under which Theorem~\ref{them_1st_countability} can be strengthened to state that every definable directed set admits a definable cofinal curve. From now on, we write that $\RR$ \textquotedblleft expands an ordered field" when the conclusion of Fact~\ref{fact:edmundo} holds.



\begin{corollary}\label{cor_cofinal_curves}
The structure $\RR$ has a pole if and only if every definable downward (respectively upward) directed set $(\Omega,\lleq)$ admits a definable downward (respectively upward) cofinal curve. 
\end{corollary} 

\begin{proof}
Let $(\Omega,\lleq)$ be a definable downward (respectively upward) directed set. If $\gamma_0:(\dom,\leq)\rightarrow (\RRR,\lle)$ is a downward cofinal curve and $\gamma_1:(\RRR,\lle)\rightarrow (\Omega,\lleq)$ is a downward (respectively upward) cofinal map then $\gamma=\gamma_1 \circ \gamma_0$ is a downward (respectively upward) cofinal curve in $(\Omega,\lleq)$. Since such a downward (respectively upward) cofinal map $\gamma_1$ exists by Theorem~\ref{them_1st_countability}, in order to prove the corollary it suffices to show that $\RR$ has a pole if and only if $(\RRR,\lle)$ admits a definable downward cofinal curve. As usual, from now on throughout the proof we omit the word downward as there is no room for confusion. 

Suppose that $\RR$ contains a pole $f:I \rightarrow J$, with $I$ a bounded interval and $J$ an unbounded interval in $R$.
By o-minimality, we may assume that $f$ is strictly monotonic. 
By transforming $f$ if necessary by means of the group operation, 
we may assume that $I=(0,r)$ for some $r>0$ and that $\lim_{t\rightarrow 0^+} f(t)=\infty$.
Continue $f$ to $(0,\infty)$ by setting $f$ to be zero on $[r,\infty)$.
For some fixed $a>0$, define $f_a \colon (0,\infty) \to R$ by $f_{a}(t) = \max \{ f(t), a \}$.
The curve $\gamma$ given by $\gamma(t) = (t, f_a(t))$ for all $t>0$ is clearly a definable cofinal curve in $(\RRR,\lle)$.

Conversely, suppose that $\gamma$ is a definable cofinal curve in $(\RRR,\lle)$ and let $\Gamma=\gamma[\dom]$. 
Let $\pi:R^2\rightarrow R$ denote the projection onto the first coordinate. Since $\gamma$ is cofinal, we have $(0,s_1)\subseteq \pi(\Gamma)$ for some $s_1 > 0$. Since $\Gamma$ is one-dimensional, by the fiber lemma for o-minimal dimension there exists $s_2>0$ such that, for every $0<s<s_2$, the fiber $\Gamma_s$ is finite. Let $s_0=\min\{s_1, s_2\}$ and consider the definable map $\mu:(0,s_0)\rightarrow R$ given by $\mu(s)=\max \Gamma_s$. This map is well defined. Since $\gamma$ is cofinal for every $s',t'>0$, there exists $s,t>0$, with $s\leq s'$ and $t\geq t'$, such that $(s,t)\in \Gamma$. It follows that $\lim_{s\rightarrow 0^+} \mu(s)=\infty$. By o-minimality, there exists some interval $(0,r)\subseteq (0,s_0)$ such that $\mu|_{(0,r)} : (0,r)\rightarrow (\mu(r),\infty)$ is a pole.
\end{proof}
\begin{remark} 
The proof above shows that, if there exists a definable one-dimensional preordered set $(\Sigma,\lleq_\Sigma)$ and a definable downward cofinal map $\gamma_{\Sigma}:(\Sigma,\lleq_\Sigma)\rightarrow (\RRR,\lle)$, then $\RR$ has a pole. Hence Theorem~\ref{them_1st_countability} cannot be \textquotedblleft improved" by putting a one-dimensional space in place of $(\RRR,\lle)$, unless we assume the existence of a pole in the structure, in which case by the above corollary $(\Sigma,\lleq_\Sigma)$ can always be taken to be $(\dom, \leq)$.   
\end{remark}


By Fact~\ref{fact:edmundo}, the following is an immediate consequence of Corollary~\ref{cor_cofinal_curves}. 

\begin{corollary}\label{cor_cofinal_curves_2}
Every definable downward (respectively upward) directed set $(\Omega,\lleq)$ admits a definable downward (respectively upward) cofinal curve if and only if $\RR$ expands an ordered field.
\end{corollary}

\section{Tame Extensions}\label{section_tame_pairs} 

In this section we use the theory of tame pairs as well as results from previous sections to obtain results about definable families of sets with the FIP.
 We conclude with a strengthening of Corollary~\ref{cor_cofinal_curves_2} in the case that $\RR$ expands an archimedean field (Corollary~\ref{cor_main_result_archimedean_field_case}). 
In this section all types are considered to be over $R$, i.e. complete type means complete as a type over $R$. 

We first fix some conventions. Let $\SSS$ be a family of definable subsets of $R^n$ with the FIP.
For every $S\in \SSS$, let $\phi_S(v,u_S)$ be a formula with parameters $u_S$ such that 
\[
S=\{v\in R^n : \RR\models \phi_S(v,u_S)\}.
\]
Then $\SSS$ may be seen as the $n$-type $p_\SSS(v)$ given by 
\[
p_\SSS(v)=\{ \phi_S(v,u_S) : S\in\SSS\}.
\] 
Conversely any type can be identified with a family of definable sets with the FIP and, in this sense, a complete $n$-type is a maximal family of definable subsets of $R^n$ with the FIP (i.e. an ultrafilter in the boolean algebra of definable subsets of $R^n$). In this section we treat $n$-types interchangeably as consistent families of formulas (with parameters) with $n$ free variables and as families of definable subsets of $R^n$ with the FIP.

A tame extension $\RR^*=(R^*,\ldots)$ of $\RR$ is a proper elementary extension of $\RR$ such that  $\{s\in R : s<t\}$ has a supremum in $R \cup \{-\infty, +\infty\}$ (i.e. the cut is definable) for every $t\in R^*$.
Note that, if $\RR$ expands the ordered set of the reals, then every proper elementary extension of $\RR$ is tame. 

An $n$-type $p(v)$ is definable if and only if, for every $(n+m)$-formula $\phi(v,u)$, the set $\{u\in R^m : \phi(v,u)\in p(v)\}$ is definable. 

Let $\xi$ be a positive element in an elementary extension of $\RR$ which is less than every positive element of $R$ (i.e. $\xi$ is infinitesimal with respect to $R$).
 Then $\RR(\xi) = (R(\xi),\ldots)$, the prime model over $R \cup \{\xi\}$, is a tame extension of $\RR$. 
Every tuple of elements in $R(\xi)$ is of the form $\gamma(\xi)$ for some curve $\gamma$ definable in $\RR$. 
By o-minimality, it holds that, for every $(n+m)$-formula $\phi(v,u)$ and every $u\in R^m$, $\RR(\xi)\models \phi(\gamma(\xi),u)$ if and only if $\RR\models \phi(\gamma(t),u)$ for all sufficiently small $t>0$. Hence every type realized in $\RR(\xi)$ is definable.  

The following is due to Marker and Steinhorn~\cite{mark_stein_94} (Theorem 2.1):
\begin{fact}\label{fact:MS}
If $\RR^*$ is a tame extension of $\RR$ and $v\in (R^*)^n$ then the type of $v$ over $R$ is definable. 
\end{fact}

Given a tame extension $\RR^*$ of $\RR$, a tame pair $(\RR^*,R)$ is the expansion of $\RR^*$ by a unary predicate defining $R$. Van den Dries and Lewenberg proved in~\cite{dries_lewenberg_95} that, if $\RR$ expands an ordered field, then the theory of such tame pairs is complete. In other words, suppose that $\RR$ expands an ordered field and let $T$ be the theory of $\RR$. Then, if $\RR_0$, $\RR_0^*$, $\RR_1$ and $\RR_1^*$ are models of $T$ such that $\RR_0^*$ is a tame extension of $\RR_0$ and $\RR_1^*$ is a tame extension of $\RR_1$, then the tame pairs $(\RR_0^*,R_0)$ and $(\RR_1^*,R_1)$ are elementarily equivalent. 

Given two families $\SSS_0$ and $\SSS_1$ of sets we say that $\SSS_1$ is finer than $\SSS_0$ if every element of $\SSS_0$ contains an element of $\SSS_1$. The main result of this section is the following. 

\begin{theorem}\label{them_tame_pairs}
Suppose that $\RR$ expands an ordered field, $\Omega \subseteq R^m$, and $\SSS=\{S_u : u\in \Omega\}$ is a definable family of subsets of $R^n$ with the FIP. The following are equivalent. 
\begin{enumerate}[(1)]
\item \label{itm:tame_pairs_1} $\SSS$ can be extended to a complete definable type in $S_n(R)$.
\item \label{itm:tame_pairs_2} There exists a definable downward directed family $\SSS'$ that is finer than $\SSS$.  
\item \label{itm:tame_pairs_3} There exists a definable curve $\gamma:\dom \rightarrow \cup\SSS$ such that, for every $u\in\Omega$, $\gamma(t) \in S_u$ for sufficiently small $t > 0$. 
\end{enumerate}
\end{theorem}
\begin{proof}
It is easy to show that (\ref{itm:tame_pairs_3}) implies (\ref{itm:tame_pairs_1}) and (\ref{itm:tame_pairs_2}). Namely, let $\gamma$ be as in (\ref{itm:tame_pairs_3}), and consider the definable downward directed family $\{\gamma[(0,t)] : t>0\}$.
This family is finer than $\SSS$, from which we conclude (\ref{itm:tame_pairs_2}). Moreover, to derive (\ref{itm:tame_pairs_1}), let $\RR(\xi)$ be an elementary extension of $\RR$ by an infinitesimal element. Then the element $\gamma(\xi)\in R(\xi)^n$ belongs in the interpretation in $\RR(\xi)$ of $S_u$ for every $u\in \Omega$, and so the complete definable $n$-type $\text{tp}(\gamma(\xi) | R)$ extends $\SSS$. Moreover $(\ref{itm:tame_pairs_2}) \Rightarrow (\ref{itm:tame_pairs_3})$ is Corollary~\ref{cor_cofinal_curves_2}. It therefore remains to show that $(\ref{itm:tame_pairs_1})$ implies $(\ref{itm:tame_pairs_2})$ or $(\ref{itm:tame_pairs_3})$. We show that (\ref{itm:tame_pairs_1}) implies (\ref{itm:tame_pairs_3}). The main tool will be the completeness of the theory of tame pairs. 

Suppose that $\SSS$ can be extended to a complete definable $n$-type $\SSS'$ over $R$. If $\SSS'$ is realised in $\RR$ it suffices to let $\gamma$ be a constant curve mapping to the realisation of $\SSS'$. Hence we assume that $\SSS'$ is not realised. Let $\RR^*$ be a (necessarily proper) elementary extension of $\RR$ such that there exists an element $c\in (R^*)^n$ that realizes $\SSS'$, and let $\RR(c)$ be the prime model over $R\cup\{c\}$. Then $\RR \lleq \RR(c)$. Note that any definable set in $\RR(c)$ is $R\cup\{c\}$-definable. It follows that $\RR(c)$ is a tame extension of $\RR$.
The point $c$ realizes $\SSS'$ in $\RR(c)$ and so, for every $u\in\Omega$, the element $c$ belongs in the interpretation of $S_u$ in $\RR(c)$. Now let $\psi(u)$ be a formula defining $\Omega$ (by adding parameters to the language if necessary we may assume that $\Omega$ is $0$-definable) and $\phi(v,u)$ is a formula such that, for every $u\in \Omega$, $S_u=\{v\in R^n : \RR\models \phi(v,u)\}$, then $(\RR(c),\RR)$ satisfies the following sentence in the language of tame pairs:
\[
\exists v \forall u ((\psi(u) \cap u\in R^m) \rightarrow \phi(v,u)). 
\]  
Let $\RR(\xi)$ be an elementary extension of $\RR$ generated 
by an infinitesimal element with respect to $\RR$. By the completeness of the theory of tame pairs, there exists $\gamma(\xi)\in \RR(\xi)$ such that $\RR(\xi)\models \phi(\gamma(\xi),u)$ for every $u\in\Omega$, where $\gamma$ is a definable (in $\RR$) curve. It follows that, for every $u\in\Omega$, $\gamma(t)\in S_u$ for all $t>0$ small enough. Hence we conclude (\ref{itm:tame_pairs_3}).    
\end{proof}

\begin{definition}
We call a type $p(v)$ a \emph{$\phi$-type}, where $\phi(v,u)$ is a formula, if there exists a set of parameters $\Omega\subseteq R^{|u|}$ such that $p(v)=\{\phi(v,u) : u\in \Omega\}$.

We call a type $p$ a \emph{type basis} if it is a filter basis, i.e. if it is downward directed. We say that a type $p(v)$ has a $\phi$-basis, for some formula $\phi$, if its restriction to $\phi$, namely $p \upharpoonright \phi =\{\phi(v,u)\in p\}$, is a type basis that generates $p$, i.e. if for any formula $\psi(v,w)$ with parameters $w$, $\psi(v,w)\in p$ if and only if there are parameters $u$ such that $\phi(v,u)\in p$ and $\RR\models \forall v (\phi(v,u)\rightarrow \psi(v,w))$. 
\end{definition}

Note that any complete type is a type basis. 
Any definable family of sets with the FIP is precisely a definable $\phi$-type for some formula $\phi$ and vice versa. Note that a type with a $\phi$-basis is definable if and only if its restriction to $\phi$ is definable. The proof of the implication $(\ref{itm:tame_pairs_3})\Rightarrow (\ref{itm:tame_pairs_1})$ in Theorem~\ref{them_tame_pairs} involves showing precisely that, by o-minimality, for any definable curve $\gamma$, the definable $\phi$-type $\{ \gamma[(0,t)] : t>0\}$ is a type basis that generates a complete definable type. 

\begin{remark}\label{remark_after_them_tame_pairs} 
As noted in the proof of Theorem~\ref{them_tame_pairs}, it follows easily from o-minimality that (\ref{itm:tame_pairs_3}) in the theorem implies (\ref{itm:tame_pairs_1}) and (\ref{itm:tame_pairs_2}), and the crux of the result lies in proving that (\ref{itm:tame_pairs_1}) implies (\ref{itm:tame_pairs_3}) and (\ref{itm:tame_pairs_2}) implies (\ref{itm:tame_pairs_3}). These are really two separate results, the former following from the completeness of the theory of tame pairs and the latter from results in previous sections. We note that only the implication $(\ref{itm:tame_pairs_1})\Rightarrow (\ref{itm:tame_pairs_3})$ requires the assumption that $\SSS$ is a $\phi$-type for some formula $\phi$ instead of just a type (and even then we do not require that the type be definable, i.e. that the index set $\Omega$ be definable). 
\end{remark}

The dimension of a type $p$ is defined to be the minimum dimension among sets in $p$. O-minimality implies that every one-dimensional complete type has a $\phi$-basis for some formula $\phi$. The same is trivially true for zero-dimensional complete types. 

The following is easy to prove. 
\begin{lemma}
If $\RR$ expands an ordered field and $\SSS$ is a type, then (\ref{itm:tame_pairs_3}) in Theorem~\ref{them_tame_pairs} is equivalent to any of the following statements. 
\begin{enumerate}[(1)] 
\item $\SSS$ is realized in $\RR(\xi)$, where $\xi$ is an element infinitesimal with respect to $\RR$.
\item $\SSS$ extends to a complete definable type of dimension at most $1$ (with a $\phi$-basis, for some formula $\phi$).
\end{enumerate}
\end{lemma}

It follows that the implications $(\ref{itm:tame_pairs_1}) \Rightarrow (\ref{itm:tame_pairs_3})$ and $(\ref{itm:tame_pairs_2})\Rightarrow (\ref{itm:tame_pairs_3})$ in Theorem~\ref{them_tame_pairs} can be formulated in terms of types as follows. 

\begin{proposition}
Let $p$ be a $\phi$-type, for some formula $\phi$. Suppose that either of the following holds.  
\begin{enumerate}[(i)]
\item $p$ can be extended to a complete definable type. 
\item $p$ can be extended to a definable type with a $\psi$-basis for some formula $\psi$. 
\end{enumerate}
Then $p$ extends to a complete definable type of dimension at most $1$.  
\end{proposition}

If we only assume that $\RR$ expands an ordered group then we may prove a result similar to the equivalence  $(\ref{itm:tame_pairs_2})\Leftrightarrow (\ref{itm:tame_pairs_3})$ in Theorem~\ref{them_tame_pairs} where we substitute (\ref{itm:tame_pairs_3}) with the existence of a definable map $\gamma: \RRR \rightarrow \cup\SSS$ where the definable family of sets $\{ \gamma[(0,s)\times (t,\infty)] : s,t>0\}$ is finer than $\SSS$. Note that this family is a basis for a type of dimension at most $2$.

From this observation and the implication $(\ref{itm:tame_pairs_2})\Rightarrow (\ref{itm:tame_pairs_3})$ in Theorem~\ref{them_tame_pairs} we derive the next corollary.

\begin{corollary}\label{cor:dim_types}
Every definable type with a $\phi$-basis may be extended to a complete definable type of dimension at most $2$.
In particular, every complete definable type with a $\phi$-basis has dimension at most $2$.
If $\RR$ expands an ordered field then every definable type with a $\phi$-basis extends to a complete definable type of dimension at most $1$ and every complete definable type with a $\phi$-basis has dimension at most $1$.
\end{corollary}

\begin{remark}\label{remark:type_2_dimensional}
The bounds in Corollary~\ref{cor:dim_types} are tight. In particular if $\RR$ does not expand an ordered field then the family $\{(0,s)\times (t,\infty) : s,t>0\}$ is a basis for a ($2$-dimensional) complete type. To see this recall Fact~\ref{fact:edmundo} and note that, if $\RR$ does not have a pole, then any cell $(f,g)$ that intersects every box of the form $(0,s)\times (t,\infty)$ must satisfy that $f=+\infty$ and $\lim_{t\rightarrow 0^+} g(t) < \infty$, and so it contains one such box. By cell decomposition it follows that, for every definable set $X\subseteq R^2$, there are $s_X, t_X>0$ such that $(0,s_X)\times (t_X,\infty)$ is either contained in or disjoint from $X$.
\end{remark}


From the fact that all complete types in an o-minimal expansion of $(\mathbb{R},<)$ are definable we may derive, using the implication $(\ref{itm:tame_pairs_1})\Rightarrow (\ref{itm:tame_pairs_3})$ in Theorem~\ref{them_tame_pairs}, the next corollary for o-minimal expansions of ordered archimedean fields. 

\begin{corollary}\label{cor_main_result_archimedean_field_case}
Suppose that $\RR$ expands an ordered field. The following are equivalent: 
\begin{enumerate}[(1)] 
\item \label{itm:cor_arch_field_1} $\RR$ is archimedean. 
\item \label{itm:cor_arch_field_2} If $\SSS$ is a definable family of sets with the FIP, then there exists a definable curve \mbox{$\gamma:\dom\rightarrow \cup\SSS$} such that, for every $S \in \SSS$, $\gamma(t)\in S$ for sufficiently small $t > 0$.
\item \label{itm:cor_arch_field_3} Every definable family $\SSS$ with the FIP can be extended to a complete definable type.  
\end{enumerate}
\end{corollary}

\begin{proof}
The implication (\ref{itm:cor_arch_field_2})$\Rightarrow$(\ref{itm:cor_arch_field_3}) follows from o-minimality. We prove (\ref{itm:cor_arch_field_3})$\Rightarrow$(\ref{itm:cor_arch_field_1}) and (\ref{itm:cor_arch_field_1})$\Rightarrow$(\ref{itm:cor_arch_field_2}). 

To show (\ref{itm:cor_arch_field_3})$\Rightarrow$(\ref{itm:cor_arch_field_1}), suppose that $\RR$ is nonarchimedean and let $r\in R$ denote an infinitesimal with respect to $1$. Consider the definable family of sets $\SSS=\{(0,1)\setminus (a-r, a+r) : a\in [0,1]\}$. Because $r$ is infinitesimal this family has the FIP. It cannot however be extended to a complete definable type $p(v)$, since in that case the set $\{t: (t \leq v)\in p(v)\}$ would have a supremum $s\in [0,1]$, contradicting $((v< s-r)\vee (s+r < v))\in p(v)$.

To show (\ref{itm:cor_arch_field_1})$\Rightarrow$(\ref{itm:cor_arch_field_2}), suppose that $\RR$ expands an archimedean field and let $\SSS=\{S_u\subseteq R^n : u\in \Omega\}$ be a definable family with the FIP. The result amounts to the following claim: that, for some $(1+n+m)$-formula $\phi(t,v,w)$, there are parameters $b\in R^m$ such that $\phi(t,v,b)$ defines a curve $\gamma$ in $R^n$ (given by $t\mapsto v$) with the property that, for all $u\in \Omega$, $\gamma(t)\in S_u$ whenever $t>0$ is small enough. So, in order to complete the proof, it suffices to show this in an elementary extension of $\RR$.  

Laskowski and Steinhorn~\cite{las_stein_95} proved that any o-minimal expansion of an archimedean ordered group is elementarily embeddable into an o-minimal expansion of the additive ordered group of real numbers. Moreover note that the fact that the family $\SSS$ has the FIP is witnessed by countably many sentences in the elementary diagram of $\RR$ and so the same property holds for the family $\SSS$ interpreted in any elementary extension of $\RR$. Therefore, by Laskowski-Steinhorn \cite{las_stein_95}, we may assume that $\RR$ expands the ordered additive group of real numbers. 

By Dedekind completeness of the reals, every elementary extension of $\RR$ is tame, and thus by Marker-Steinhorn (Fact~\ref{fact:MS}) every type over $R$ is definable. In particular any expansion of $\SSS$ to a complete type over $R$ is definable. The corollary then follows from Theorem~\ref{them_tame_pairs}.
\end{proof}

Our study of cofinality via the theory of tame pair highlights, through the equivalence $(\ref{itm:tame_pairs_1})\Leftrightarrow (\ref{itm:tame_pairs_2})$ in Theorem~\ref{them_tame_pairs}, a connection between definable types and downward directed families of sets in o-minimal expansions of fields. In work in preparation~\cite{andujar19}, the first author extends this connection further, proving that this equivalence holds in all o-minimal structures, and develops the topological applications of this.
 
\begin{theorem}[\hspace{1sp}\cite{andujar19}]
Let $\SSS$ be a definable family of subsets of $R^n$ with the FIP. The following are equivalent. 
\begin{enumerate}[(1)]
\item $\SSS$ can be extended to a complete definable $n$-type over $R$. 
\item There exists a definable downward directed family that is finer than $\SSS$.
\end{enumerate}
\end{theorem}

\section{Definable topological spaces}\label{section_def_top_spaces}
Our understanding of a \textquotedblleft definable topological space" is that of a topological space in which the topology is explicitly definable in a sense that generalises Flum and Ziegler~\cite{flum_ziegler_80} and Pillay~\cite{pillay87}. 
We restrict this idea however to the specific context of an o-minimal structure, wanting to expand on similar work around metric space topology by the third author~\cite{walsberg15}. We study one dimensional definable topological spaces in detail in~\cite{atw2}. Our definition is featured independently by Johnson~\cite{johnson14}, and Peterzil and Rosel~\cite{pet_rosel_18}. 


\begin{definition}\label{def_definable_top_spaces}
A \emph{definable topological space} $(X,\tau)$ is a topological space such that $X\subseteq R^n$ and there exists a definable family of sets $\BB\subseteq \mathcal{P}(X)$ that is a basis for $\tau$. 
\end{definition}


As in any o-minimal structure the canonical definable topology in $\RR$ is the order topology in $R$ and induced product topology in $R^n$. We refer to these topologies as euclidean.

Note that, given a definable topological space, a definable basis of neighborhoods of a given point is a definable downward directed family.

The following definition of convergence of a curve in a definable topological space is a natural adaptation of a standard notion in the area of o-minimality, and in particular of recent similar definitions in~\cite{thomas12} and~\cite{walsberg15}.

\begin{definition}
Let $(X,\tau)$ be a definable topological space, $x\in X$ and $\gamma:\dom\rightarrow X$ be a curve. 

We say that $\gamma$ \emph{converges in $(X,\tau)$ (or $\tau$-converges) to $x$ as $t\rightarrow 0$}, and write $\tau$-$\lim_{t\rightarrow 0} \gamma=x$, if, for every neighborhood $A$ of $x$, there exists $t(A)>0$ such that $\gamma(t)\in A$ whenever $0<t \leq t(A)$. 

Analogously we say that \emph{$\gamma$ converges in $(X,\tau)$ (or $\tau$-converges) to $x$ as $t\rightarrow \infty$}, and write $\tau$-$\lim_{t\rightarrow \infty} \gamma=x$, if, for every neighborhood $A$ of $x$, there exists $t(A)>0$ such that $\gamma(t)\in A$ whenever $t\geq t(A)$.
\end{definition}   

As usual we omit references to the topology when there is no possible ambiguity.    
When making the assumption that $\RR$ expands a field we will always understand convergence of a curve $\gamma$ to mean convergence as $t\rightarrow 0$. 


Recall the classical definitions of the \emph{density}, $d(X)$, and \emph{character}, $\chi(X)$, of a topological space $(X,\tau)$, namely 
$$d(X)=\min\{|D| : D\subseteq X \text{ is dense in } X\}$$ and 
$$\chi(X)=\sup_{x\in X}\min\{|\BB| : \BB \text{ a basis of neighborhoods of $x$}\}.$$
 In particular $X$ is separable if and only if $d(X)\leq \aleph_0$, and first countable if and only if $\chi(X)\leq \aleph_0$. Given a definable set $Y$ and a lack of reference to a definable topology on $Y$ we write $d(Y)$ to denote the density of $Y$ with the euclidean topology. 

The first result in this section is an observation that follows only from o-minimality, that is, we do not require the assumption that $\RR$ expands an ordered group. It implies that any definable topological space in an o-minimal expansion of $(\R,<)$ is first countable.   

\begin{proposition}\label{prop_1st_countability}
\sloppy Every definable downward directed set admits a downward cofinal subset of cardinality at most $d(R)$.

It follows that every definable topological space $(X,\tau)$ satisfies \mbox{$\chi(X) \leq d(R)$}. In particular, if $(R,<)$ is separable, then any definable topological space is first countable.
\end{proposition}

Note that this proposition cannot be improved by writing second countable in place of first countable, since the discrete topology on $R$ is definable and not second countable whenever $R$ is uncountable.

\begin{proof}[Proof of Proposition~\ref{prop_1st_countability}]
First we recall some facts of o-minimality. Let $D_R\subseteq R$ be a set dense in $R$ and let $C\subseteq R^n$ be a cell of dimension $m$. There exists a projection $\pi:C\rightarrow R^{m}$ that is a homeomorphism onto an open cell $C'$. If $m=0$ then $d(C)=1$. Otherwise the set $\pi^{-1}(D_R^m\cap C')$, where $D_R^m=D_R\times \overset{m}{\cdots}\times D_R$, is clearly dense in $C$. By o-minimal cell decomposition, it follows that any definable set $B$ satisfies $d(B)\leq d(R)$. Secondly recall that, by the frontier dimension inequality, for any definable set $B$ and definable subset $B'$, if $\dim B =\dim B'$ then $B'$ has interior in $B$.     

Now set $\lambda = d(R)$.  Let $(\Omega,\lleq)$ be a definable downward directed set. For any $u\in \Omega$, let $S_u=\{v\in\Omega : v\lleq u \}$. By Remark~\ref{remark_sets_order} the definable family $\SSS=\{S_u : u\in \Omega\}$ is downward directed. We fix $S_*\in \SSS$ such that $\dim S_*=\min\{\dim S : S\in\SSS\}$, and let $D$ be a dense subset of $S_*$ of cardinality at most $\lambda$. We claim that $D$ is downward cofinal in $\Omega$. This is because, for every $S_u\in \SSS$, there exists $S'\in\SSS$ such that $S'\subseteq S_u\cap S_*$, and so $\dim (S_u\cap S_*)=\dim(S_*)$, and in particular $S_u\cap S_*$ has interior in $S_*$. This means that, for every $u\in \Omega$, there exists some $v\in D$ such that $v\in S_u\cap S_*$, and so $v\lleq u$.

Let $(X,\tau)$ be a definable topological space. For any given $x\in X$, let $\BB(x)=\{A_u : u\in\Omega_x\}$ be a definable basis of neighborhoods of $x$. Let $(\Omega_x,\lleq_x)$ denote the definable downward directed set given $u\lleq_x v \Leftrightarrow A_u\subseteq A_v$. Note that, if $D$ is a downward cofinal subset of $\Omega_x$, then the family $\{A_u : u\in D\}$ is still a basis of neighborhoods of $x$ in $(X,\tau)$. This completes the proof of the proposition.  
\end{proof}

We now make use of Theorem~\ref{them_1st_countability} and Corollary~\ref{cor_cofinal_curves_2} to prove a fundamental theorem regarding definable topological spaces.
It might strike the reader as surprising that we name this theorem \textquotedblleft definable first countability", as opposed to using that label for Proposition~\ref{prop_1st_countability}. The reasons for this will be explained later in the section, but in short it is due to the fact that a consequence (Proposition~\ref{curve_selection}) of Theorem~\ref{cor_main_cofinal_bases} is that, whenever $\RR$ expands an ordered field, any definable topological space admits definable curve selection. 





\begin{theorem}[Definable first countability]\label{cor_main_cofinal_bases}
Let $(X,\tau)$ be a definable topological space with definable basis $\BB$.
\begin{enumerate}[(1)]
\item \label{itm:1st_countability_1} For every $x\in X$, there exists a definable basis of neighborhoods of $x$ of the form $\{A_{(s,t)} : s,t>0 \} \subseteq \BB$ satisfying that, for any neighborhood $A$ of $x$, there exist $s_A, t_A >0$ such that $A_{(s,t)}\subseteq A$ whenever $(s,t)\lle (s_A,t_A)$.  

\item \label{itm:1st_countability_2} Suppose that $\RR$ expands an ordered field. Then, for every $x\in X$, there exists a definable basis of neighborhoods of $x$ of the form $\{A_{t} : t>0\} \subseteq \BB$ satisfying that, for any neighborhood $A$ of $x$, there exists $t_A>0$ such that $A_{t}\subseteq A$ whenever $0<t\leq t_A$.  
\end{enumerate}
\end{theorem}
\begin{proof}
Let $(X,\tau)$ be a definable topological space with definable basis $\BB=\{A_u : u\in\Omega\}$. We begin by proving (\ref{itm:1st_countability_1}). 

Given $x\in X$ let $\Omega_x=\{u\in\Omega : x\in A_u\}$. The family $\{A_u : u\in \Omega_x\}\subseteq \BB$ is a definable basis of neighborhoods of $x$ that induces a definable preorder $\lleq_\BB$ by inclusion (see Remark~\ref{remark_sets_order}(\ref{itm:remark_i})) making $(\Omega_x,\lleq_{\BB})$ into a definable downward directed set. By Theorem~\ref{them_1st_countability}, let $\gamma:(\RRR,\lle)\rightarrow (\Omega_x, \lleq_\BB)$ be a definable downward cofinal map. The family $\{A_{\gamma(s,t)} : s,t>0\}$ has the desired properties. 

The proof of (\ref{itm:1st_countability_2}) follows analogously from Corollary~\ref{cor_cofinal_curves_2}. 
\end{proof}

We denote the neighborhood bases described in (\ref{itm:1st_countability_1}) and (\ref{itm:1st_countability_2}) in the above theorem \emph{cofinal bases of neigborhoods of $x$ in $(X,\tau)$}. By using Corollaries~\ref{cor_improvem_them} and~\ref{cor_cofinal_curves_2}, one may show that these may be chosen uniformly on $x\in X$.


The next corollary is a refinement of Theorem~\ref{cor_main_cofinal_bases}. It shows in particular that the cofinal bases of neighborhoods described in (\ref{itm:1st_countability_2}) may be assumed to be nested.

\begin{corollary}\label{cor_nested_cofinal_bases}
Let $(X,\tau)$ be a definable topological space. 
\begin{enumerate}[(1)]
\item \label{itm:cor_nested_bases_1} For every $x\in X$, there exists a definable basis of open neighborhoods of $x$ of the form $\{A_{(s,t)} : s,t>0 \}$ satisfying $A_{(s',t')}\subseteq A_{(s,t)}$ whenever $(s',t')\lle (s,t)$. 

\item \label{itm:cor_nested_bases_2} Suppose that $\RR$ expands an ordered field. Then, for any $x\in X$, there exists a definable basis of open neighborhoods of $x$ of the form \mbox{$\{ A_t : t>0\}$} satisfying $A_s \subseteq A_t$ whenever $0<s<t$. 
\end{enumerate}
\end{corollary}
\begin{proof}
We prove~(\ref{itm:cor_nested_bases_2}), and then sketch how statement~(\ref{itm:cor_nested_bases_1}) follows in a similar fashion.

\sloppy Let $(X,\tau)$ be a definable topological space and suppose that $\RR$ expands an ordered field. Let $x\in X$ and \mbox{$\A=\{ A_t : t>0\}$} be a cofinal basis of neighbourhoods of $x$ as given by Theorem~\ref{cor_main_cofinal_bases} (\ref{itm:1st_countability_2}). Let $f:(0,\infty)\rightarrow (0,\infty)$ be the function defined as follows. For every $t>0$, $f(t)=\sup\{s<1 : A_{s'}\subseteq A_t \text{ for every } 0<s'<s\}$. By definition of $\A$ and o-minimality, this function is well defined and definable. By o-minimality, let $(0,r_0]$ be an interval on which $f$ is continuous. Note that the family $\{A_t : 0<t\leq r_0\}$ is still a basis of neighborhoods of $x$. By continuity, for every $0<r<r_0$, $f$ reaches its minimum in $[r, r_0]$, so there exists $t_r>0$ such that $A_{t_r}\subseteq \cap_{r\leq t \leq r_0} A_t$. In particular the set $\cap_{r\leq t \leq r_0} A_t$ remains a neighborhood of $x$. For every $t>0$, consider the definable set 
\[
A'_t=\begin{cases}\bigcap_{t\leq t' \leq r_0} A_{t'} & \text{ if } t<r_0; \\ A_{r_0} &\text{ otherwise. }\end{cases}
\]  
The family $\{\textrm{int}_\tau (A'_t) : t>0\}$ is a definable basis of open neighborhoods of $x$ such that, for every $0<s<t$, $\textrm{int}_\tau (A'_s) \subseteq \textrm{int}_\tau (A'_t)$.  

To prove~(\ref{itm:cor_nested_bases_1}) note first that, if $\RR$ expands an ordered field, then we may use the construction in~(\ref{itm:cor_nested_bases_2}) taking $A_{(s,t)}=A_s$. Suppose that $\RR$ does not expand an ordered field (equivalently, by Fact~\ref{fact:edmundo}, $\RR$ does not have a pole). Let $x\in X$ and let $\{A_{(s,t)} : s,t>0 \}$ be a basis of neighborhoods of $x$ as given by Theorem~\ref{cor_main_cofinal_bases}~(\ref{itm:1st_countability_1}). Using definable choice, let $f$ be a definable map on $\RRR$ such that $A_{(s',t')}\subseteq A_{(s,t)}$ for every $s',t'>0$ with $(s',t')\lle f(s,t)$. By o-minimality let $C$ be a cell, cofinal in $(\RRR,\lle)$, where $f$ is continuous. Since $\RR$ does not have a pole there exists some box $(0,r_0]\times [r_1,\infty)$ contained in $C$ (see Remark~\ref{remark:type_2_dimensional}). Following the approach to proving~(\ref{itm:cor_nested_bases_2}), let $A'_{(s,t)}=\bigcap \{A_{(s',t')} : (s,t)\lle (s',t')\lle (r_0,r_1)\}$ if $(s,t)\lle (r_0,r_1)$, and $A'_{(s,t)}=A_{(r_0,r_1)}$ otherwise. Our desired basis is given by $\{ \textrm{int}_\tau(A'_{(s,t)}) \, s,t>0\}$. 
\end{proof}


Recall that in the field of topology a net is a map from a (generally upward) directed set into a topological space for which there is a notion of convergence. The utility of nets lies in that they completely capture the topology in the following sense: the closure of a subset $Y$ of a topological space $(X,\tau)$ is the set of all limit points of nets in $Y$. It follows that nets encode the set of points on which a given function between topological spaces is continuous. 

It seems natural to define a \emph{definable net} to be a definable map from a definable directed set (which we can choose to be downward definable directed set for convenience) into a definable topological space. Indeed, applying definable choice one may show that definable nets defined in this way have properties that are equivalent to those of nets in general topology, when considering definable subsets of definable topological spaces and definable functions between definable topological spaces.  

Theorem~\ref{cor_main_cofinal_bases} suggests that, for all practical purposes, it is enough to consider definable nets to be those whose domain is \mbox{$(\RRR,\lle)$}, if $\RR$ expands an ordered group but does not expand an ordered field, and to simply consider them as definable curves if $\RR$ does expand an ordered field (see Propositions~\ref{curve_selection} and~\ref{prop_cont_lim} below). This is analogous to the way in which, in first countable spaces, for all purposes nets can be taken to be sequences. It is from this last observation, namely that, whenever $\RR$ expands an ordered field, definable curves take the role in definable topological spaces of sequences in first countable spaces, that the motivation to label Theorem~\ref{cor_main_cofinal_bases} \textquotedblleft definable first countability" arose.

This argument of course implicitly relates definable curves in the definable o-minimal setting to sequences in general topology. This identification however is far from new in the o-minimal setting. It is mentioned for example by van den Dries in~\cite{dries98}, p.93, with respect to the euclidean topology in the context of o-minimal expansion of ordered groups, and implicitly noted by the second author in~\cite{thomas12} when introducing the notions of \emph{definable Cauchy curves} and closely related \emph{definable completeness} in the context of certain topological spaces of definable functions where the underlying o-minimal structure expands an ordered field. The same implicit identification can be seen in the definition of definable compactness in~\cite{pet_stein_99}, which is given in terms of convergence of definable curves and applied to euclidean spaces; we will explore this correspondence further in the next section. The motivation for this correspondence lies precisely in that, in the settings being considered by the authors, definable curves display properties similar to those of sequences in the corresponding setting of classical topology. 

We conclude this section with results illustrating that, whenever $\RR$ expands an ordered field, definable curves indeed take the role of definable nets, displaying properties similar to those of sequences in first countable topological spaces. In particular we show that definable topological spaces admit a weak version of definable curve selection (one for which definable curves are not necessarily continuous), and consequently that continuity of definable functions can be characterized in terms of convergence of definable curves.

 
\begin{proposition}[Definable curve selection]\label{curve_selection}
Suppose that $\RR$ expands an ordered field. Let $(X,\tau)$ be a definable topological space and let $Y\subseteq X$ be a definable set. Then $x\in X$ belongs to the closure of $Y$ if and only if there exists a definable curve $\gamma$ in $Y$ that converges to $x$. 
\end{proposition}
\begin{proof}
Let $(X,\tau)$ be a definable topological space and $x\in X$. Let $Y\subseteq X$ be a definable set and suppose that $x\in cl(Y)$. If $\RR$ expands an ordered field there exists, by Theorem~\ref{cor_main_cofinal_bases}, a definable basis of neighborhoods of $x$, $\{A_t : t>0\}$, such that, for every neighborhood $A$ of $x$, $A_t$ is a subset of $A$ for all $t>0$ small enough. 

Given one such definable basis of neighborhoods, consider by definable choice a definable curve $\gamma$ such that, for every $t>0$, $\gamma(t)\in A_t \cap Y$. This curve clearly lies in $Y$ and converges in $(X,\tau)$ to $x$. 

Conversely it follows readily from the definition of curve convergence that, if there exists a curve in $Y$ converging to $x\in X$, then $x\in cl(Y)$.  
\end{proof}

\begin{proposition}\label{prop_cont_lim}
Suppose that $\RR$ expands an ordered field. Let $(X,\tau_X)$ and $(Y,\tau_Y)$ be definable topological spaces. Let $f:(X,\tau_X) \rightarrow (Y,\tau_Y)$ be a definable map. Then, for any $x\in X$, $f$ is continuous at $x$ if and only if, for every definable curve $\gamma$ in $X$, if $\gamma$ $\tau_X$-converges to $x$ then $f \circ \gamma$ $\tau_Y$-converges to $f(x)$. 
\end{proposition}
\begin{proof}
Suppose that $f$ is continuous at $x$ and let $\gamma$ be a definable curve in $X$ $\tau_X$-converging to $x\in X$. By continuity of $f$, for every neighborhood $A$ of $f(x)$, $f^{-1}(A)$ is a neighborhood of $x$, so for all $t$ small enough $\gamma(t)\in f^{-1}(A)$. This means that, for all such values of $t$, $(f\circ \gamma)(t)\in A$, so $f\circ \gamma$ $\tau_Y$-converges to $f(x)$.

Conversely, suppose that $f$ is not continuous at $x\in X$. Thus there exists some neighborhood $A$ of $f(x)$, which we may assume is definable, such that $f^{-1}(A)$ is not a neighborhood of $x$. By Proposition~\ref{curve_selection}, there exists a definable curve $\gamma$ in $X\setminus f^{-1}(A)$ that $\tau_X$-converges to $x$. For every $t>0$, it holds that $(f\circ \gamma)(t)\notin A$, and so $f\circ \gamma$ does not $\tau_Y$-converge to $f(x)$. 
\end{proof}

\section{Definable compactness}\label{section_compactness}

In this final section we work with a notion of definable compactness for definable topological spaces that generalises the standard notion in the o-minimal setting~\cite{pet_stein_99}, as well as the more recent definition in~\cite{walsberg15}. In the sense that definable curves can be taken in the o-minimal setting 
to represent sequences, this definition can be understood to interpret sequential compactness, rather than compactness. However, making use of results from previous sections, we'll show how this property relates to what appear to be stronger notions, all of which could be considered reasonable alternative approaches to the idea of compactness in the definable setting.   


\begin{definition} \label{def_definable_curve_comp}
\sloppy Let $(X,\tau)$ be a definable topological space and let \mbox{$\gamma:\dom\rightarrow X$} be a curve. We say that $\gamma$ is \emph{$\tau$-completable} if both $\tau$-$\lim_{t\rightarrow 0} \gamma$ and $\tau$-$\lim_{t\rightarrow \infty} \gamma$ exist.

A definable topological space $(X,\tau)$ is \emph{definably compact} if every definable curve in it is $\tau$-completable.
\end{definition}   
As usual we omit references to the topology when there is no possible ambiguity. Since we are in the ordered group setting, our definition is equivalent to the one that would follow from a less restrictive notion of curve (see~\cite{pet_stein_99}).


Recall that, for convenience, in this paper we are approaching the theory of nets in terms of downward instead of upward directed sets. Consequently a net $\gamma:(\Omega,\lleq)\rightarrow (X,\tau)$ converges to $x\in X$ if, for every neighborhood $A$ of $x$, there exists $u_A\in \Omega$ such that $\gamma(u)\in A$ whenever $u\lleq u_A$. Recall that a subnet (a Kelley subnet) of $\gamma$ is a net of the form $\gamma'= \gamma \circ f$ where $f:(\Omega', \lleq')\rightarrow (\Omega, \lleq)$ is a downward cofinal map. It is definable if all of $(\Omega',\lleq')$, $f$ and $\gamma$ are definable. It is a classical result in topology that a space is compact if and only if every net in it admits a convergent subnet. 

We now apply Corollary~\ref{cor_cofinal_curves_2} to provide characterizations of definable compactness in terms of definable nets and in terms of definable downward directed families of closed sets. The latter part of this characterization ($(\ref{itm:compactness_3})$ in the next corollary) has already been considered in the o-minimal setting. In particular, Johnson~\cite{johnson14} uses it as the definition of definable compactness.

\begin{corollary}\label{characterize_def_compactness}
Let $(X,\tau)$ be a definable topological space. The following are equivalent. 
\begin{enumerate}[(1)]
\item \label{itm:compactness_1} $(X,\tau)$ is definably compact in the sense of Definition \ref{def_definable_curve_comp}, i.e. every definable curve in $X$ is completable.
\item \label{itm:compactness_2} Every definable net in $X$ has a convergent definable subnet. 
\item \label{itm:compactness_3} Every definable downward directed family of closed subsets of $X$ has nonempty intersection. 
\end{enumerate}
\end{corollary}
\begin{proof}
We prove $(\ref{itm:compactness_2})\Rightarrow (\ref{itm:compactness_1})$, $(\ref{itm:compactness_3})\Rightarrow (\ref{itm:compactness_2})$ and $(\ref{itm:compactness_1})\Rightarrow (\ref{itm:compactness_3})$, in that order. 

$(\ref{itm:compactness_2})\Rightarrow (\ref{itm:compactness_1})$. Let $\gamma:\dom \rightarrow X$ be a definable curve in $X$. The map $\gamma$ is a definable net from $(\dom,\leq)$ to $(X,\tau)$. 
Suppose that $\gamma$ admits a definable subnet $\gamma'=\gamma\circ f$ with cofinal map $f:(\Omega, \lleq)\rightarrow (\dom ,\leq)$ converging to a point $x\in X$. We show that $\gamma$ converges to $x$ as $t\rightarrow 0$. Let $A$ be a neighborhood of $x$, which we may assume is definable. 
By the definition of net convergence, there exists $u_A\in \Omega$ such that, for every $v\lleq u_A$, $(\gamma \circ f)(v)\in A$. Moreover, since $f$ is cofinal, for every $t>0$ there exists $u_t\in \Omega$ such that $f(v)\leq t$ for every $v\lleq u_t$. Let $u\lleq \{u_A,u_t\}$. Then $(\gamma\circ f)(u)\in A$ and $f(u)<t$. By o-minimality, it follows that the definable set $\{ t>0 : \gamma(t)\in A\}$ contains an interval of the form $(0,t_0)$. So $\lim_{t\rightarrow 0}\gamma(t)=x$. The proof that $\gamma$ converges as $t\rightarrow \infty$ is analogous, only that we interpret $\gamma$ to be a net from $(\dom, \geq)$ to $(X,\tau)$.

$(\ref{itm:compactness_3})\Rightarrow (\ref{itm:compactness_2})$. This mimics the classical proof of the analogous implication in general topology. Let $\gamma:(\Omega,\lleq)\rightarrow (X,\tau)$ be a definable net.
Define
$$C_u=\gamma(\{v\in \Omega: v\lleq u\}),$$
for all $u \in \Omega$.
 Since $(\Omega,\lleq)$ is a downward directed set, the definable family $\CC=\{cl_\tau(C_u) : u\in \Omega\}$ is downward directed. Suppose that (\ref{itm:compactness_3}) holds, in which case there exists $x\in \cap \CC$. We show that there exists a definable subnet that converges to $x$. Let $\A=\{A_u : u\in \Sigma\}$ be a definable basis of neighborhoods of $x$. As usual we understand $\A$ with inclusion as a definable directed set by identifying each set in the family with its index parameters. Set $\Omega'=\{(u,v)\in\Omega\times\Sigma : \gamma(u)\in A_v\}$. Since $x\in\cap \CC$, it follows that

\begin{equation}\label{eqn_compactness_nets} \tag{$\ddagger$}
(\forall u\in \Omega)(\forall A\in \A)(\exists u'\in \Omega)(u'\lleq u \wedge \gamma(u')\in A).
\end{equation}
In particular, for any $v\in \Sigma$, there exists $u\in \Omega$ such that $(u,v)\in\Omega'$. 

Consider the definable preordered set $(\Omega', \lleq')$, where $\lleq'$ is given by 
$$(u',v')\lleq'(u,v) \Leftrightarrow (u'\lleq u) \wedge (A_{v'}\subseteq A_v).$$
For any pair $(u,v), (u',v')\in \Omega'$, there exists $w\in \Omega$ such that $w\lleq \{u, u'\}$ and $v''\in \Sigma$ such that $A_{v''}\subseteq A_{u}\cap A_{u'}$. By~\eqref{eqn_compactness_nets}, there exists $u''\lleq w \lleq \{u,u'\}$ such that $\gamma(u'')\in A_{v''}$. So $(u'',v'')\in \Omega'$ and $(u'',v'')\lleq' \{(u,v), (u',v')\}$, meaning that $(\Omega',\lleq')$ is a directed set. Let $\pi:\Omega'\rightarrow \Omega$ be the corresponding projection and $\gamma'=\gamma\circ \pi$. By~\eqref{eqn_compactness_nets}, $\pi$ is clearly cofinal. We claim that the definable subnet $\gamma':(\Omega',\lleq')\rightarrow (X,\tau)$ converges to $x$. To see this let $A_v$, $v\in \Sigma$, be a basic open neighborhood of $x$. 
By~\eqref{eqn_compactness_nets}, 
let $u\in\Omega$ be such that $(u,v)\in \Omega'$. If $(u',v')\lleq' (u,v)$ then it holds that $\gamma'(u',v')=\gamma(u')\in A_{v'}\subseteq A_{v}$. This proves the claim. 

$(\ref{itm:compactness_1})\Rightarrow (\ref{itm:compactness_3})$. Suppose that $(X,\tau)$ is definably compact and let $\CC=\{C_u : u\in \Omega\}$ be a definable downward directed family of closed sets. Following Remark~\ref{remark_sets_order}(\ref{itm:remark_i}), let $(\Omega, \lleq_\CC)$ be the definable preordered set induced by set inclusion in $\CC$. By definition of downward directed family of sets, $(\Omega,\lleq_\CC)$ is a downward directed set. By Theorem~\ref{them_1st_countability}, there exists a definable cofinal map $\gamma:(\RRR,\lle) \rightarrow (\Omega,\lleq_\CC)$.

By means of definable choice, let $\mu:\RRR \rightarrow X$ be a definable map such that $\mu(s,t)\in C_{\gamma(s,t)}$ for every $s,t>0$. 

Now, for every $s>0$, consider the fiber $\mu_s:R^{>0}\rightarrow X$ where $\mu_s(t)=\mu(s,t)$. By definable compactness, there exists $x_s\in X$ such that $\lim_{t\rightarrow \infty} \mu_s = x_s$. Note that the map $s\mapsto x_s$ for $s>0$ is definable (if the space is not Hausdorff then we may apply definable choice to select one limit uniformly). By definable compactness, let $x^*\in X$ be the limit of $x_s$ as $s\rightarrow 0$. We show that $x^*\in \cap\CC$.

Fix $u\in \Omega$. By cofinality of $\gamma$, there exists $(s_u,t_u)\in\RRR$ such that $\mu(s,t)\in C_u$ for every $(s,t)\lle (s_u,t_u)$. Hence, for any $0<s \leq s_u$, the image of the curve $\mu_s$ for $t\geq t_u$ belongs in $C_u$ and so, since $C_u$ is closed, the limit point $x_s$ belongs in $C_u$. Since this is true for all $0<s\leq s_u$, we conclude that $x^*\in C_u$. So $x^*\in \cap \CC$.   
\end{proof}

Note that (\ref{itm:compactness_3}) in Corollary~\ref{characterize_def_compactness} can be formulated in terms of open sets as follows: every upward directed definable family $\A$ of open sets that is a cover of $X$ admits a finite subcover (equivalently satisfies $X\in \A$). This property is a notion of compactness (it clearly resembles compactness in terms of families of closed sets with the FIP) in the definable setting that has been treated in recent years by Johnson~\cite{johnson14} and in unpublished work by Fornasiero~\cite{fornasiero}. Condition (\ref{itm:compactness_2}) was included in order to highlight how, in analogy to general topology, condition (\ref{itm:compactness_3}) relates to what would be a notion of definable compactness in terms of nets. The equivalence $(\ref{itm:compactness_2})\Leftrightarrow(\ref{itm:compactness_3})$ can be proved by adapting to our setting the classical proof of the similar result in general topology. This can be done relying solely on definable choice, dismissing o-minimality or the existence of ordered group structure. The most important part of the corollary therefore lies in showing that, in the o-minimal group setting, the apparently weaker condition (\ref{itm:compactness_1}) implies (\ref{itm:compactness_2}) and (\ref{itm:compactness_3}). 

Suppose that in our setting we understood our definition of definable compactness (every definable curve is completable) to interpret sequential compactness, and the definition given by Johnson in~\cite{johnson14} (every definable downward directed family of closed sets has nonempty intersection) to interpret compactness. Then we can understand Corollary~\ref{characterize_def_compactness} as asserting that, in the definable realm and context of an underlying o-minimal expansion of an ordered group, sequential compactness and compactness are equivalent.

The first author shows in forthcoming work~\cite{andujar19} that these two notions of definable compactness are not equivalent in general for definable topologies in the o-minimal setting (although one may easily show their equivalence for one-dimensional spaces), and proves their equivalence for Hausdorff spaces.



The following example shows that we may not relax the \textquotedblleft downward directed" condition in Corollary~\ref{characterize_def_compactness} (\ref{itm:compactness_3}) to simply \textquotedblleft having the FIP". 
\begin{example}\label{example_FIP}
By passing to an elementary extension if necessary, we assume that $\RR$ is non-archimedean. Let $r$ denote an infinitesimal element in $\RR$ with respect to another element $1$. 
Then
$$ \mathcal{S} = \{ [0,1]\setminus (x-r, x+r) : x \in [0,1] \} $$
is a definable family of closed sets with FIP but with empty intersection. 

Moreover note that, as observed in the proof of Corollary~\ref{cor_main_result_archimedean_field_case} ($(\ref{itm:cor_arch_field_3})\Rightarrow (\ref{itm:cor_arch_field_1})$), $\mathcal{S}$ cannot be extended to a complete definable type $p(v)$, since in that case the set $\{t: (t \leq v)\in p(v)\}$ would have a supremum $s\in [0,1]$, contradicting $((v\leq s-r)\vee (s+r \leq v))\in p(v)$. 
\end{example}


Finally we show that, if $\RR$ expands an ordered field and $\SSS$ is a definable family of closed sets with the FIP that extends to a complete definable type, then definable compactness implies that $\SSS$ has nonempty intersection. We frame this idea in the next corollary. The proof relies on Theorem~\ref{them_tame_pairs}.

\begin{definition}
Let $(X,\tau)$ be a definable topological space, let $p\in S_X(R)$ be a type and let $x\in X$. We say that $\{x\}$ is a \emph{specialization} (or $x$ is a \emph{limit}) of $p$ if, for every closed definable set $B\subseteq X$, if $B\in p$ then $x\in B$. Equivalently, if $p$ is complete, we have that $A\in p$ for every definable neighborhood $A$ of $x$.  
\end{definition}

\begin{corollary}\label{deflim_compactness}
Suppose that $\RR$ expands an ordered field. Let $(X,\tau)$ be a definable topological space. The following are equivalent. 
\begin{enumerate}[(1)]
\item \label{itm:field_compactness_1} $(X,\tau)$ is definably compact. 
\item \label{itm:field_compactness_2} For every complete definable type $p\in S_X(R)$, there exists $x\in X$ such that $\{x\}$ is a specialization of $p$. 
\end{enumerate}
\end{corollary}
\begin{proof}
Let $(X,\tau)$, $X\subseteq R^n$, be a definable topological space. We start by showing that $(\ref{itm:field_compactness_2})$ implies $(\ref{itm:field_compactness_1})$. Let $\gamma:\dom \rightarrow X$ be a definable curve, which we assume is injective. Since $\RR$ has field structure, it suffices to show that $\gamma$ converges at $t\rightarrow 0$. Consider the definable type 
\begin{multline*}
p=\{ \phi(v,u): \phi \text{ is a formula with parameters $u$ and $n$ free variables $v$}\, \\ \textrm{ such that } (\exists t>0)\, (\forall\,0<s<t) \, \RR\models \phi(\gamma(s),u)\}.
\end{multline*}
By o-minimality, this is a complete type; in particular it is a one-dimensional complete definable type. Let $x\in X$ be such that $\{x\}$ is a specialization of $p$. For any $t>0$, clearly $x\in cl \gamma[(0,t)]$. Suppose that $\lim_{t\rightarrow 0} \gamma(t) \neq x$, in which case, by o-minimality, there exists some neighborhood $A$ of $x$ and some $t>0$ such that $A\cap \gamma[(0,t)]=\emptyset$, so $x\notin cl \gamma[(0,t)]$, which is a contradiction.

We now prove $(\ref{itm:field_compactness_1})$ implies $(\ref{itm:field_compactness_2})$. Let $p \in S_X(R)$ be a complete definable type, let $\{A_u : u\in \Omega\}$ be a definable basis for $\tau$ and let $\Omega'=\{u\in \Omega: A_u\notin p\}$. Since $p$ is definable, the set $\Omega'$ is definable. If $\Omega'=\emptyset$, then clearly any point $x\in X$ is a specialization of $p$. Suppose that $\Omega'\neq \emptyset$. For each $u\in \Omega'$, let $C_u=X\setminus A_u$. Since $p$ is complete, for every $u\in \Omega$, it holds that $C_u\in p$. By Theorem~\ref{them_tame_pairs}, there exists a definable curve $\gamma:\dom \rightarrow \cup\CC$ satisfying that, for every $u\in\Omega'$, $\gamma(t)\in C_u$ for every $t>0$ small enough. Suppose that $(X,\tau)$ is definably compact and let $x=\lim_{t\rightarrow 0} \gamma$. Then clearly $x\in C_u$ for every $u\in \Omega'$. 

Now let $A$ be a definable neighborhood of $x$. Let $u\in \Omega$ be such that $x\in A_u\subseteq A$. If $A_u\notin p$ then $u\in \Omega'$ and so $x\in C_u$, contradicting that $x\in A_u$. Hence $A_u\in p$, and so $A\in p$. We conclude that $\{x\}$ is a specialization of $p$.
\end{proof}

Let $(X,\tau)$ be a definable topological space and let $\RR^*$ be a tame extension of $\RR$, where $\RR$ expands an ordered field. Let $(X^*,\tau^*)$ be the interpretation of $(X,\tau)$ in $\RR^*$. Given $x\in X$ and $y\in X^*$ it would be reasonable to describe $y$ as being \emph{infinitely close to $x$ (with respect to $\RR$)} if, for every definable neighborhood $A\subseteq X$ of $x$, it holds that $y\in A^*$, where $A^*$ is the interpretation of $A$ in $\RR^*$. In this sense Corollary~\ref{deflim_compactness} states that $(X,\tau)$ is definably compact if and only if, in any tame extension $\RR^*$ of $\RR$, every point in $X^*$ remains infinitely close to a point in $X$.

Finally, by recalling the Marker-Steinhorn Theorem (Fact~\ref{fact:MS}), we derive from Corollary~\ref{deflim_compactness} that, whenever $\RR$ expands the field of reals, compactness is equivalent to definable compactness. 


\begin{corollary}
Suppose that $\RR$ expands the field of reals and let $(X,\tau)$ be a definable topological space. Then $(X,\tau)$ is compact if and only if it is definably compact. 
\end{corollary}
\begin{proof}
Let $(X,\tau)$ be a definable topological space. Clearly compactness implies that every definable downward directed family of closed subsets of $X$ has nonempty intersection, so by the implication $(\ref{itm:compactness_3})\Rightarrow (\ref{itm:compactness_1})$ in Corollary~\ref{characterize_def_compactness} (which is dependent on o-minimality alone) compactness implies definable compactness. We prove the converse. 

Suppose that $(X,\tau)$ is definably compact and let $\CC\subseteq \mathcal{P}(X)$ be a family of closed subsets of $X$ with the FIP. We show $\cap \CC\neq \emptyset$. 

For each $x\in X\setminus \cap\CC$, let $A_x$ be a definable open neighborhood of $x$ such that $A_x \cap C=\emptyset$ for some $C\in \CC$. Clearly $\cap \CC = \cap \{X\setminus A_x : x\in X\setminus \cap \CC\}$. So by passing to the latter family if necessary we may assume that $\CC$ contains only definable sets. Let $p$ be a complete type extending $\CC$. By the Marker-Steinhorn Theorem (Fact~\ref{fact:MS}), $p$ is definable so by Corollary~\ref{deflim_compactness} there exists $x\in X$ such that $\{x\}$ is a specialization of $p$. In particular, $x\in C$ for every $C\in \CC$.  
\end{proof}

\bibliography{mybib_directed_sets}
\bibliographystyle{alpha}

\end{document}